\documentclass[letterpaper, 10 pt]{article}  

\usepackage[total={6in, 9in}]{geometry}

\title{\LARGE \bf
Homogeneous Formulation of Convex Quadratic Programs for Infeasibility Detection}

\usepackage{graphicx}
\usepackage{amsmath}
\usepackage{amssymb}
\usepackage[ruled,vlined,linesnumbered]{algorithm2e}
\usepackage{amsthm}
\usepackage{enumerate}

\allowdisplaybreaks

\newcommand{\R}{\mathbb{R}}
\newcommand{\Sy}{\mathbb{S}}

\newcommand{\yhat}{\widehat{y}}
\newcommand{\xihat}{\widehat{\xi}}
\newcommand{\nuhat}{\widehat{\nu}}
\newcommand{\omegahat}{\widehat{\omega}}
\newcommand{\tauhat}{\widehat{\tau}}

\newcommand{\sigmamin}{\sigma^{\textnormal{min}}}
\newcommand{\sigmamax}{\sigma^{\textnormal{max}}}
\newcommand{\Diag}{\textnormal{Diag}}

\newcommand{\half}{\frac{1}{2}}

\newcommand{\setN}{{\cal N}}
\newcommand{\xbar}{\overline{x}}
\newcommand{\lambdabar}{\overline{\lambda}}
\newcommand{\sbar}{\overline{s}}

\newcommand{\ipopt}{\textsc{IPOPT}}

\newtheorem{theorem}{Theorem}
\newtheorem{assumption}{Assumption}
\newtheorem{lemma}{Lemma}
\newtheorem{remark}{Remark}
\author{Arvind U Raghunathan
\thanks{A.U. Raghunathan is with Mitsubishi Electric Research Laboratories, 201 Broadway, Cambridge, MA 02139, USA. {\tt \small raghunathan@merl.com}}%
}

\begin{document}

\maketitle

\begin{abstract}
Convex Quadratic Programs (QPs) have come to play a central role in the computation of control action for constrained dynamical systems. 
In this paper, we present a novel Homogeneous QP (HQP) formulation which is obtained by embedding the original QP in a larger space.  The key properties of the HQP are: (i) is always feasible, (ii) an optimal solution to QP can be readily obtained from a solution to HQP, and (iii) infeasibility of QP corresponds to a particular solution of HQP. An immediate consequence is that all the existing algorithms for QP are now also capable of robustly detecting infeasibility. In particular, we present an Infeasible Interior Point Method (IIPM) for the HQP and show polynomial iteration complexity when applied to HQP.  A key distinction with prior IPM approaches is that we do not need to solve second-order cone programs. Numerical experiments on the formulation are provided using existing codes.  
\end{abstract}
\section{Introduction}\label{sec:introduction}

Optimization algorithms are widely used in the control for constrained dynamic systems. In particular, the solution of convex Quadratic Programs (QPs) is a key ingredient of optimization algorithms.  Convex QPs arise for example in the Model Predictive Control of linear systems, switched linear systems and nonlinear systems~\cite{borrelli_bemporad_morari_2017,Mayne2020}.  The past two decades have witnessed the development of a number of algorithms for the solution of convex QPs arising in the context of optimization-based control. 
Initial work on QP algorithms were based on solving a system of equations representing the optimality conditions such as Interior Point Method (IPM) for QP~\cite{RaoWrightRawlings1998}, active-set methods~\cite{BartlettBiegler2006,FerreauBockDiehl2008} and more recently, IPM for Second Order Cone Programs (SOCPs)~\cite{DomahidiChuBoydECC2013}. In the last decade there has been interest in the development of first-order approaches such as gradient projection~\cite{Richter2012}, dual gradient projection~\cite{GiselssonACC2013,GPAD2014}, and splitting methods~\cite{RaghunathanACC2014,GiselssonBoydCDC2014,PatrinosStellaCDC2014,Ghadimi2015} and iterative second-order approaches~\cite{QuirynenKnyazev2018}.  Recently, there has been work on employing these methods within Branch \& Bound (B\&B) algorithms for the solution of Mixed Integer Quadratic Programs (MIQPs) that arise in MPC for switched systems~\cite{FrickDomahidi2015,NaikBemporadIFAC2017,StellatoNaikBemporadECC2018,HespanholQuirynen2019}.

A critical feature required of QP solvers for real-time applications is the ability to detect infeasibility and provide a graceful handling of the same. 
Infeasibility handling increases in importance when solving MIQPs in a B\&B setting as it is expected that QPs resulting from the fixing of a subset of binary variables will likely be infeasible.  Detecting such infeasible QPs and pruning the search tree in B\&B is essential for computational efficiency of the MIQP solver. 

Infeasibility of a QP is determined by the identification of a ray that is feasible for the dual and makes the dual objective unbounded. First-order primal-dual algorithms are capable of detecting infeasibility.  Recent results on detecting infeasibility using first-order algorithms include~\cite{RaghunathanCDC2014,Banjac2019,LiaoMcPherson2020}. A dual QP formulation, if it can be constructed explicitly, will allow the determination of a ray of infeasibility. However, such an explicit dual construction requires strong convexity of the QP.  
Active-set methods typically employ a \emph{feasibility} phase to first determine a feasible point and then proceed to an \emph{optimization} phase which produces the optimal solution. Failure of the feasibility phase allows to pronounce a QP as being infeasible~\cite{NocedalWrightBook06}. The work performed in this phase is identical to that required for the \emph{optimization} phase. When employing inexact linear algebra the first-phase is typically not employed and as a consequence the ability of such approaches to certify infeasibility is not clear. 
IPM algorithms, on the other hand, are not all capable of producing a certificate of infeasibility~\cite{WrightIPMBook,WongThesis}. IPMs based on Homogeneous Self-Dual (HSD) embedding~\cite{YeToddMizuno1994,XuHungYe1996} are known to produce a certificate of infeasibility for Linear Programs (LPs).  This technique has been implemented in MOSEK~\cite{MOSEKLP} and has also been extended to handle SOCPs. To detect infeasible QPs with IPMs, the QPs must first be formulated as a SOCP to which the HSD embedding is then applied. The ECOS solver~\cite{DomahidiChuBoydECC2013} does exactly this for the solution of QPs arising in MPC and can detect infeasibility. 

We present a novel embedding of the QP into a space that is one dimension higher. The embedding is produced by introducing an additional variable that is nonnegative and multiplies the affine terms in the objective and constraints. The objective contribution from the introduced variable is determined so the resulting Homogeneous Quadratic Program (HQP) satisfies the following properties: (i) HQP is always feasible, (ii) an optimal solution to the QP can always be recovered from the optimal solution to the HQP if the introduced variable is positive; and (iii) QP can be certified to be infeasible if the introduced variable is zero. 
More importantly, the HQP formulation can be presented to any QP solver to determine the optimal solution or infer infeasibility in a robust manner.

In this paper, we focus on applying an Infeasible IPM (IIPM) to solve the HQP. We are able to derive polynomial complexity of a standard IIPM when applied to HQP by adapting the analysis of IIPM for linear programs~\cite[Chapter~6]{WrightIPMBook}.  Thus, we have an IIPM with polynomial iteration complexity for inferring optimality or infeasibility of QP without having to solve SOCPs.

The paper is organized as follows. \S\ref{sec:qpformulation} presents the QP and the assumptions. \S\ref{sec:hqp} presents the embedding of QP into a HQP and shows the equivalence of the formulations. An IIPM for solving HQP and the polynomial iteration complexity are described in \S\ref{sec:iipm}.  We present results in \S\ref{sec:results} when using the HQP formulation with \ipopt.  Conclusions and directions for future work are provided in \S\ref{sec:conclusions}.

\emph{Notation}. The set of reals is denoted by $\R$ and the set of vectors of dimension $n$ by $\R^n$. The set of $n \times n$ symmetric matrices is denoted by $\Sy^n$. For a matrix $A \in \Sy^n$, the notation $A \succcurlyeq (\succ) 0 $ denotes that A is positive semidefinite (definite). Given a vector $u \in \R^n$, $\Diag(u)$ denotes a diagonal matrix with the elements of the vector on the diagonal. The notation $I_n$ denotes the identity matrix of size $n$.

\section{Problem Formulation}\label{sec:qpformulation}

Consider the QP in the form
\begin{subequations}
\begin{align}
    \min_{y \in \R^n}          &\; \half y^TCy + c^Ty \label{qp.obj} \\
    \text{s.t.}     &\; Ey = f \label{qp.eq} \\
                    &\; y \geq 0 \label{qp.nonneg} 
\end{align}\label{qp}
\end{subequations}
where $C \in \Sy^n$, $c \in \R^n$, $E \in \R^{m \times n}$, and $f \in \R^m$.  Note that any QP formulation with finite lower or upper bound on each of the variables can be put into the form in~\eqref{qp} through a linear transformation of variables.  

We make the following assumptions on QP~\eqref{qp}.
\begin{assumption}\label{assume:E}
The matrix $E$ has full row rank of $m$.
\end{assumption}
\begin{assumption}\label{assume:ZCZ}
The matrix $Z^TCZ \succ 0$ where $Z \in \R^{n \times (n-m)}$ is an orthonormal basis for the null-space of $E$, i.e. $Z$ is a basis for $\{ v \,|\, Ev = 0 \}$.
\end{assumption}
Assumption~\ref{assume:E} is not restrictive and can be easily satisfied by removing dependent rows if necessary. Assumption~\ref{assume:ZCZ} requires that $C$ be positive definite when projected onto $Z$. This assumption readily holds for MPC formulations~\cite{RaghunathanCDC2014} and also for certain spectral  relaxations of nonconvex MIQPs~\cite{NohraRaghunathan2021,MPB2021}.  Assumption~\ref{assume:ZCZ} implies that the optimal solution to QP~\eqref{qp} is unique whenever QP~\eqref{qp} is feasible.  Note that the assumptions do not preclude the infeasibility of QP~\eqref{qp}.  The basis $Z$ is typically computed using a QR factorization of $E^T$~\cite{GolubVanLoanBook}. 

We conclude this section with a statement of the conditions for optimality and infeasibility of QP~\eqref{qp}.

A point $y^\star$ minimizes QP~\eqref{qp} if there exist multipliers $\nu^\star \in \R^{m}$, $\xi^\star \geq 0 \in \R^n$, satisfying the first-order optimality conditions
\begin{subequations}\label{qpopt}
  \begin{align}
    Cy^\star + c + E^T \nu^\star - \xi^* &= 0 \label{qpopt.dual}  \\
    Ey^\star &= f \label{qpopt.eq} \\
    0 \leq y^\star &\perp \xi^\star \geq 0. \label{qpopt.compl} 
  \end{align}
\end{subequations}
The conditions~\eqref{qpopt} are necessary and sufficient for a minimizer of QP~\eqref{qp} under Assumption~\ref{assume:ZCZ}.

The QP~\eqref{qp} is infeasible if there exist $\nu^\circ \in \R^{m}$, $\xi^\circ \in \R^{n}$ satisfying
\begin{subequations}\label{qpinfeas}
  \begin{align}
    E^T\nu^\circ - \xi^\circ &= 0 \label{qpinf.c1} \\
    f^T\nu^\circ &= -1 \label{qpinf.c2} \\
    \xi^\circ &\geq 0. \label{qpinf.c3}
  \end{align}
\end{subequations}
The conditions in~\eqref{qpinfeas} are obtained from applying a Theorem of the Alternative~\cite{Mangasarian} to the constraints in~\eqref{qp.eq}-\eqref{qp.nonneg}.

\section{Homogeneous QP Formulation}\label{sec:hqp}

Consider embedding QP~\eqref{qp} into a Homogeneous Quadratic Program (HQP) as
\begin{subequations}\label{hqp}  
  \begin{align}
    \min\limits_{y \in \R^n,\tau \in \R} &\; \half y^TCy + \tau c^Ty + \frac{\theta}{2} (\tau^2-2\tau) \label{hqp.obj}\\
    \text{s.t.} &\; Ey = f\tau \label{hqp.eq} \\
	&\; y \geq 0, \tau \geq 0 \label{hqp.nonneg}
  \end{align}
\end{subequations}
where $\tau \in \R$ is an additional nonnegative variable and $\theta > 0$ is a parameter.  Observe that the objective~\eqref{hqp.obj} is obtained by multiplying the linear term $c^Ty$ in~\eqref{qp.obj} with $\tau$ and appending with the term $(\theta/2)(\tau^2-\tau)$.  The equality constraints~\eqref{hqp.eq} is obtained by multiplying the right-hand side of~\eqref{qp.eq} with $\tau$. \S\ref{sec:condnstheta} provides conditions on $\theta$ that ensures HQP is convex.  \S\ref{sec:computetheta} shows how the parameter $\theta$ can be computed efficiently.  Finally, \S\ref{sec:qphqp} shows that solving HQP allows to recover a solution to QP or declare infeasibility.

\subsection{Conditions on $\theta$}\label{sec:condnstheta}
The parameter $\theta$ is chosen to satisfy two conditions:
\begin{enumerate}
\item $\theta > 2|\theta^\star|$ where $\theta^\star$ is defined as
\begin{align}
\theta^\star 				= \min\limits_{y \in \R^n} &\; \half y^TCy + c^Ty \label{deftheta} \\
								  \text{s.t.} \;& Ey = f. \nonumber			
\end{align}
The parameter $\theta^\star$ is well-defined by Assumption~\ref{assume:ZCZ} and can be computed as $\theta^* = \half \tilde{y}^TC\tilde{y} + c^T\tilde{y}$ where $\tilde{y}$ is obtained by solving a single linear system
\begin{align}
    \begin{pmatrix} 
    C & E^T \\ E & 0 
    \end{pmatrix} \begin{pmatrix} \tilde{y} \\ \tilde{\nu} \end{pmatrix} 
    = \begin{pmatrix} -c \\ f \end{pmatrix}.
\end{align}
Since the nonnegativity constraints~\eqref{qp.nonneg} are ignored in the definition of $\theta^*$~\eqref{deftheta} it follows that
\begin{align}
 -\half\theta < \theta^\star \leq& \half y^TCy + c^Ty \label{theta.ineq} \\
 &\forall\, y \text{ satisfying }  \eqref{qp.eq}-\eqref{qp.nonneg}.
 \nonumber
\end{align}

\item $\theta$ is chosen large enough so that
\begin{align}
					&	\widehat{Z}^T \begin{bmatrix} 	C 			& c \\ 
						c^T 		& \theta \end{bmatrix} \widehat{Z} \succ 0  
																	\label{hqp.pdredhess} 
\end{align}
where $\widehat{Z}$ is a basis for the null space of~\eqref{hqp.eq}.  Such a basis can be obtained as ,
\begin{equation}
 	\widehat{Z} = \begin{bmatrix} Z & d \\
										0	 & 1 				\end{bmatrix} 
	\label{defZhat}
\end{equation}
where $d = E^T(EE^T)^{-1}f$.
\end{enumerate}
The satisfaction of~\eqref{hqp.pdredhess} implies that: 
\begin{enumerate}
    \item the Hessian of the objective in~\eqref{hqp.obj} is positive definite on the null space of the equality constraints~\eqref{hqp.eq}.  
    \item HQP~\eqref{hqp} is convex and the first-order optimality conditions are necessary and sufficient for a minimizer.
\end{enumerate}

\subsection{Computing $\theta$}\label{sec:computetheta}
We will now address the computation of $\theta$ satisfying~\eqref{hqp.pdredhess}.

\begin{lemma}
Suppose Assumptions~\ref{assume:E} and \ref{assume:ZCZ} hold.  Then~\eqref{hqp.pdredhess} holds for all $\theta$ satisfying 
\begin{equation}
	\theta > \frac{1}{\lambda_{\min}(Z^TCZ)}\|Z^T(Cd + c)\|^2 - d^TCd - 2 c^Td. \label{theta4pd}
\end{equation}
\end{lemma}
\begin{proof}
The claim follows from an application of the Schur-complement to the left-hand side of the matrix in~\eqref{hqp.pdredhess}.  Substitute $\widehat{Z}$ from~\eqref{defZhat} in~\eqref{hqp.pdredhess} to obtain
\begin{equation}
					\widehat{Z}^T \begin{bmatrix} 	C 			& c \\ 
						c^T 		& \theta \end{bmatrix} \widehat{Z} 
			= \begin{bmatrix} Z^TCZ & Z^T(Cd + c) \\
							(Cd + c)^TZ & \theta + d^TCd + 2c^Td \end{bmatrix}.
			\label{ZCZ.expand}
\end{equation}
Since $Z^TCZ \succ 0$ (Assumption~\ref{assume:ZCZ}) the condition in~\eqref{hqp.pdredhess} holds provided the Schur-complement w.r.t. $Z^TCZ$ in~\eqref{ZCZ.expand} is positive, i.e.
\begin{equation}
\begin{aligned}
		0 <&\; \theta + d^TCd + 2 c^Td - \\
		&\; (Cd+c)^TZ (Z^TCZ)^{-1} Z^T(Cd + c) \\
\implies \theta >&\;  (Cd+c)^TZ (Z^TCZ)^{-1} Z^T(Cd + c) \\
        &\; - d^TCd - 2 c^Td. 
\end{aligned}\label{defthetamin}
\end{equation}
Using $Z^TCZ \succeq \lambda_{\min}(Z^TCZ) I_{n-m}$ in~\eqref{defthetamin} yields~\eqref{theta4pd}.
\end{proof}

\begin{remark}\label{rem:noZ}
It is not necessary to compute $Z^T(Cd+c)$ explicitly.  Since $\|ZZ^T\| \leq 1$ it is sufficient to choose 
\begin{equation}
	\theta > \frac{1}{\lambda_{\min}(Z^TCZ)}\|Cd + c\|^2 - d^TCd - 2 c^Td \label{theta4pd.a}
\end{equation}
in order to satisfy~\eqref{hqp.pdredhess}.
\end{remark}

\begin{remark}\label{rem:nolammin}
Further, it is not necessary to compute $\lambda_{\min}(Z^TCZ)$.  It is sufficient to have a nontrivial lower bound estimate of $\lambda_{\min}(Z^TCZ)$.  For example, if $C \succ 0$ (as in MPC applications) then $\lambda_{\min}(C) \leq \lambda_{\min}(Z^TCZ)$.  In the context of MIQPs with binary variables, QPs~\eqref{qp} are solved at each node of the B\&B tree.  The QP solved at the child node is a result of fixing a binary variable to $0$ or $1$.  Let us suppose that in the child node the first variable index is fixed to 0 or 1. Then
\[
\lambda_{\min}(Z^TCZ) = \min\limits_{u : Eu = 0} \frac{u^TCu}{u^Tu} \leq \min\limits_{u : Eu = 0, u_1 \in \{0 , 1\}} \frac{u^TCu}{u^Tu}.
\]
Hence, the lower bound estimate available at the parent node serves as a valid lower bound estimate for smallest eigenvalue of the reduced Hessian at the child node.     If $0 < \alpha \leq \lambda_{\min}(Z^TCZ)$ then the choice of  
\begin{equation}
	\theta > \frac{1}{\alpha}\|Cd + c\|^2 - d^TCd - 2 c^Td \label{theta4pd.b}
\end{equation}
ensures satisfaction of~\eqref{hqp.pdredhess}.
\end{remark}

\subsection{Equivalence between QP and HQP}\label{sec:qphqp}
We collect some simple observations on the HQP~\eqref{hqp}. 
\begin{enumerate}[(O1)]
\item \label{obs.feas} HQP~\eqref{hqp} is always feasible.  It is easily verified that 
$(y,\tau) = 0$ satisfies \eqref{hqp.eq}-\eqref{hqp.nonneg}.  
\item \label{obs.opt} HQP~\eqref{hqp} has an optimal solution with optimal value less than or equal to $0$. This follows directly from~(O1).
\item \label{obs.compact} HQP~\eqref{hqp} has a finite optimum.  This follows from~\eqref{hqp.pdredhess}.
\end{enumerate}

We now state the first-order optimality conditions for HQP~\eqref{hqp}.

A point $(\yhat,\tauhat)$ minimizes HQP~\eqref{hqp} if there exist multipliers $(\nuhat,\xihat,\omegahat) \in \R^{m+n+1}$ 
satisfying the first-order optimality conditions
\begin{subequations}\label{hqpopt}
  \begin{align}
    C\yhat + \tauhat c + E^T\nuhat - \xihat &= 0 \label{hqpopt.dual1}  \\
    \theta \tauhat - \theta + c^T\yhat - f^T\nuhat - \omegahat &= 0 \label{hqpopt.dual2} \\
    E\yhat &= f\tauhat \label{hqpopt.eq} \\
    0 \leq \yhat &\perp \xihat \geq 0 \label{hqpopt.compl} \\
    0 \leq \tauhat &\perp \omegahat \geq 0. \label{hqpopt.sigma}
  \end{align}
\end{subequations}

\begin{theorem}\label{thm:optimal}
Suppose $\theta$ is chosen to satisfy the conditions in \S\ref{sec:condnstheta}. 
The QP~\eqref{qp} has an optimal solution $y^\star$ iff the HQP~\eqref{hqp} 
has an optimal solution $(\yhat,\tauhat)$ with $\tauhat > 0$.  
\end{theorem}
\begin{proof}
Consider the if part.  Let $(\yhat,\tauhat)$ be an optimal solution of HQP~\eqref{hqp} and let $(\nuhat,\xihat,\omegahat)$ be multipliers such that~\eqref{hqpopt} holds.  Then, it is easily verified that $(y^\star,\nu^\star,\xi^\star) = (\yhat/\tauhat, 
\nuhat/\tauhat,\xihat/\tauhat)$ satisifies the optimality 
conditions~\eqref{qpopt} for QP~\eqref{qp}.  This proves the if part.

Consider the only if part.  Let $y^\star$ be an optimal solution of QP~\eqref{qp} and let $(\nu^\star,\xi^\star)$ be multipliers such that~\eqref{qpopt} holds.  Define $\bar{\tau} = \theta/(\theta+c^Ty^\star-f^T\nu^\star)$.  If $\bar{\tau} > 0$ then it can be easily verified that $\yhat = \bar{\tau} y^\star$, $\tauhat = \bar{\tau}$, $\nuhat = \bar{\tau} \nu^\star$, $\xihat = \bar{\tau} \xi^\star$, $\omegahat = 0$ satisfy the optimality conditions for HQP~\eqref{hqpopt}.  To show  $\bar{\tau} > 0$ we need to show that $\theta+c^Ty^\star-f^T\nu^\star > 0$ since $\theta > 0$.  Consider
\begin{subequations}
\begin{align}
    &\; \theta + c^Ty^\star - f^T\nu^\star \label{thetahat.c1}\\
    =&\; \theta + c^Ty^\star - (E^T\nu^\star)^Ty^\star \label{thetahat.c2} \\ 
	=&\; \theta + 2c^Ty^\star + (y^\star)^TC y^\star - (y^\star)^T\xi^\star \label{thetahat.c3} \\
	=&\; \theta + 2c^Ty^\star + (y^\star)^TC y^\star \label{thetahat.c4} \\
	>&\; 0 					
\end{align}
\end{subequations}
where the equality in~\eqref{thetahat.c2} follows by multiplying~\eqref{qpopt.eq} by $(\nu^\star)^T$ and substituting for $f^T\nu^\star$ with $(E^T\nu^\star)^Ty^\star$. Multiplying~\eqref{qpopt.dual} by $(y^\star)^T$ and substituting for $- (E^T\nu^\star)^Ty^\star$ as $c^Ty^\star + (y^\star)^TC y^\star - (y^\star)^T\xi^\star$ yields~\eqref{thetahat.c3}.  Using the complementarity constraints~\eqref{qpopt.compl} in~\eqref{thetahat.c3} yields~\eqref{thetahat.c4}. The final inequality follows from~\eqref{theta.ineq}. 
This proves the only if part of the claim. The claim is proven.
\end{proof}

We now show that infeasibility of QP~\eqref{qp} is equivalent to the vanishing of the optimal solution to HQP~\eqref{hqp}.
\begin{theorem}\label{thm:infeasible}
Suppose $\theta$ satisfies the conditions in \S\ref{sec:condnstheta}.  The QP~\eqref{qp} is infeasible iff the HQP~\eqref{hqp} has optimal solution $(\yhat,\tauhat) =0$.
\end{theorem}
\begin{proof}
Consider the only if part of the claim. 
Suppose there exists $(\nu^\circ,\xi^\circ)$ satisfying~\eqref{qpinfeas}. It can be verified that $(\yhat,\tauhat) = 0$, $(\nuhat,\xihat,\omegahat) = (\theta\nu^\circ,\theta\xi^\circ,0)$ satisfies~\eqref{hqpopt}.  Hence $(\yhat,\tauhat) = 0$ is an optimal solution to HQP~\eqref{hqp}. This proves the only if part of the claim.

Consider the if part of the claim. Suppose $(\yhat,\tauhat) = 0$ is the optimal solution to~\eqref{hqp} and let $(\nuhat,\xihat,\omegahat)$ be the multipliers in~\eqref{hqpopt}. Then $(\nu^\circ,\xi^\circ) = (\nuhat/(\theta+\omegahat),\xihat/(\theta+\omegahat))$ can be verified to satisfy~\eqref{qpinfeas}. This proves the if part of the claim, completing the proof.
\end{proof}

\section{Infeasible Interior Point Method (IIPM)}\label{sec:iipm}

For the remainder of the paper, we assume that the HQP has the form
\begin{subequations}
\begin{align}
    \min\limits_{x \in \R^{n+1}} &\; \half x^TQx + q^Tx \label{hqp1.obj} \\
    \text{s.t.} &\; Ax = 0 \label{hqp1.eq} \\
    &\; x \geq 0 \label{hqp1.nonneg}
\end{align}\label{hqp1}
\end{subequations}
where $x = \left(\begin{smallmatrix} y \\ \tau \end{smallmatrix}\right)$, $Q = \left(\begin{smallmatrix} C & c^T \\ c & \theta \end{smallmatrix}\right)$, $q = \left(\begin{smallmatrix} 0 \\ -\theta \end{smallmatrix}\right)$, $A = \left(\begin{smallmatrix} E & -f \end{smallmatrix}\right)$. We assume the following for the HQP~\eqref{hqp1}. 
\begin{assumption}\label{assume:A}
The matrix $A$ has full row rank of $m$.
\end{assumption}
\begin{assumption}\label{assume:ZQZ}
The matrix $Q$ is positive definite on the null space of $A$, i.e. $x^TQx > 0 \,\forall\, x \in \{ v \,|\, Av = 0\}$.
\end{assumption}
The above HQP can be identified with original QP~\eqref{qp} when $f = 0$ and the Assumptions~\ref{assume:E}-\ref{assume:ZCZ} imply satisfaction of Assumptions~\ref{assume:A}-\ref{assume:ZQZ}. If $f \neq 0$ \S\ref{sec:hqp} shows how to cast~\eqref{qp} satisfying Assumptions~\ref{assume:E}-\ref{assume:ZCZ} into a HQP of the form in~\eqref{hqp1} satisfying Assumptions~\ref{assume:A}-\ref{assume:ZQZ}. 

In the rest of the section, we show how a standard IIPM for LPs can be extended to HQP~\eqref{hqp1}.  The IIPM also enjoys polynomial iteration complexity. 

\subsection{IIPM for HQP}\label{sec:iipm4hqp}

In the rest of the section we employ the functions
\begin{subequations}
\begin{align}
    r_d(x,\lambda,s) :=&\; Qx + A^T \lambda - s \label{hqp1.rd} \\
    r_p(x,\lambda,s) :=&\; Ax \label{hqp1.rp} \\
    r_c(x,\lambda,s) :=&\; Xs \label{hqp1.rc} \\
    \mu(x,s) :=&\; x^Ts/(n+1) \label{hqp1.mu}
\end{align}\label{hqp1.r}
\end{subequations}
where $x \in \R^{n+1}$, and $\lambda \in \R^m$, $s \in \R^{n+1}$ are the multipliers for the equality, nonnegativity constraints~\eqref{hqp1.eq}-\eqref{hqp1.nonneg}, and $X = \Diag(x)$.  The parameter $\mu(x,s)$ is called a barrier or centrality parameter.  In the following, we will suppress the arguments when the dependence is clear from the context.  

A point $x^\star \geq 0$ is said to minimize HQP~\eqref{hqp1} if there exist $\lambda^\star$ and $s^\star \geq 0$ satisfying the first-order stationary conditions
\begin{align}
(r_d(x,\lambda,s),r_p(x,\lambda,s),r_c(x,\lambda,s)) = (0,0,0).
\label{hqp1opt}
\end{align}
IPMs aim to compute $(x^\star,\lambda^\star,s^\star)$ by following the central path which is defined by $\{ (x,\lambda,s) \,|\, (x,s) > 0,\,  (r_d(x,\lambda,s),r_p(x,\lambda,s),r_c(x,\lambda,s)) = (0,0,\mu(x,s) e) \}$ where $e \in \R^{n+1}$ is a vector of all ones. In the limit as $\mu(x,s) \rightarrow 0$ we recover a point satisfying~\eqref{hqp1opt}.  
Feasible IPMs typically assume that an initial iterate $(x,\lambda,s)$ satisfying $(x,s) > 0$ and $(r_d(x,\lambda,s),r_p(x,\lambda,s)) = (0,0)$.  Such a point is not generally guaranteed to exist and can also be difficult to compute.  In the context of HQP~\eqref{hqp1} such a point will not exist if the original QP~\eqref{qp} is infeasible (refer Theorem~\ref{thm:infeasible}). This is our motivation for considering IIPM.  

IIPMs start from an initial iterate $(x^0,\lambda^0,s^0)$ satisfying $(x^0, s^0) > 0$ but $(r_d^0,r_p^0) \neq 0$ where $r_d^0, r_p^)$ denote  $r_d(x^0,\lambda^0,s^0)$, $r_p(x^0,\lambda^0,s^0)$. The method generates a sequence of iterates $\{(x^k,\lambda^k,s^k)\}$ where the iterates are required to lie within a neighborhood $\setN_{-\infty}(\gamma,\beta)$ with
\begin{align}
    \setN_{-\infty}(\gamma,\beta) = \left\{ (x,\lambda,s) \,\left|\, 
    \begin{aligned}
    \frac{\|(r_d,r_p)\|}{\mu} \leq \beta \frac{\|(r_d^0,r_p^0)\|}{\mu^0} \\
    (x,s) > 0, 
    Xs \geq \gamma \mu(x,s)e  
    \end{aligned}
    \right. \right\} \label{defnhd}
\end{align}
where $\beta \geq 1 $ is a fixed parameter. 
At each iteration a search direction $(\Delta x^k,\Delta \lambda^k,\Delta s^k)$ is generated by solving
\begin{align}
    \begin{pmatrix}
        Q & A^T & -I_{n+1} \\
        A & 0 & 0 \\
        S^k & 0 & X^k 
    \end{pmatrix} \begin{pmatrix}
        \Delta x^k \\
        \Delta \lambda^k \\
        \Delta s^k
    \end{pmatrix} = 
    \begin{pmatrix}
        -r_d^k \\
        -r_p^k \\
        -r_c^k + \sigma^k\mu^k e
    \end{pmatrix} \label{iipm.step}
\end{align}
where $\sigma^k > 0$ such that $\sigma^k \in [\sigmamin,\sigmamax]$ and $\mu^k = (x^k)^Ts^k/(n+1)$. The residual of the complementarity condition $r_c^k$ is perturbed by $\sigma^k \mu^k$ to ensure that the iterates $(x^{k+1},s^{k+1})$ can be guaranteed to lie within the neighborhood $\setN_{-\infty}(\gamma,\beta)$. Given such a direction define $(x(\alpha),\lambda(\alpha),s(\alpha)) = (x^k,\lambda^k,s^k) + \alpha(\Delta x^k,\Delta \lambda^k,\Delta s^k)$. The step $\alpha^k$ is chosen to be the largest value such that 
\begin{subequations}
\begin{align}
        (x(\alpha^k),\lambda(\alpha^k),s(\alpha^k)) \in&\;  \setN_{-\infty}(\gamma,\beta) \label{nhdk1} \\ 
        \mu(\alpha^k) \leq&\; (1 - 0.01\alpha^k)\mu^k. \label{suffdeck1}
\end{align}\label{condsk1}
\end{subequations}
The next iterate is defined as $(x^{k+1},\lambda^{k+1},s^{k+1}) = (x(\alpha^k),\lambda(\alpha^k),s(\alpha^k))$. Algorithm~\ref{algo:iipm} describes the IIPM.

\begin{algorithm}\label{algo:iipm}
\KwData{$\beta,\gamma,\sigmamin,\sigmamax$ with $\beta \geq 1$, $\gamma \in (0,1)$, $0 < \sigmamin < \sigmamax \leq \half$. } 
Choose $(x^0,\lambda^0,s^0)$ with $x^0, s^0 > 0$.\\
\For{$k = 0,1,\ldots$}{
    Choose $\sigma^k \in [\sigmamin,\sigmamax]$ and solve~\eqref{iipm.step} to obtain $(\Delta x^k,\Delta \lambda^k,\Delta s^k)$. \\
    Choose the largest $\alpha^k > 0$ such that~\eqref{condsk1} holds. \\
    Set $(x^{k+1},\lambda^{k+1},s^{k+1}) = (x^k,\lambda^k,s^k) + \alpha^k(\Delta x^k,\Delta \lambda^k,\Delta s^k)$.
}
\caption{IIPM for HQP~\eqref{hqp1}}
\end{algorithm}

\subsection{Polynomial Complexity}\label{sec:polycompiipm}

The analysis of the Algorithm~\ref{algo:iipm} is straightforward adaption of the analysis in~\cite[Chapter~6]{WrightIPMBook}.  We will only highlight the key steps where we differ.  The results are specialized to the case where the initial iterate is chosen as 
\begin{align}
    (x^0,\lambda^0,s^0) = (\zeta e,0,\zeta e) \label{initpoint}
\end{align} 
where $\zeta > 0$ is scalar for which 
\begin{align}
    \|(x^\star,s^\star)\| \leq \zeta \label{defzeta}
\end{align}
for some optimal solution $(x^\star,\lambda^\star,s^\star)$ of HQP~\eqref{hqp1}.  Recall from the discussion in~\S\ref{sec:hqp} that such a solution always exists under the Assumptions~\ref{assume:A}-\ref{assume:ZQZ}.

From the linearity of the residuals in~\eqref{hqp1.rd}-\eqref{hqp1.rp}, the choice of step direction~\eqref{iipm.step} and the method of updating the iterate at $(k+1)$ it is easy to show that
\[
    (r_d^{k+1},r_p^{k+1}) = (1 - \alpha^k)(r_d^k,r_c^k) = \upsilon^{k+1}(r_d^0,r_p^0)
\]
where $\upsilon^{k+1} = \prod_{j=0}^k(1-\alpha^j)$.
We first state a key result that is used in subsequent lemmas.
\begin{lemma}\label{lemma.1}
Suppose $(\xbar,\lambdabar,\sbar)$ be such that $r_d(\xbar,\lambdabar,\sbar) = 0$ and $r_p(\xbar,\lambdabar,\sbar) = 0$.  Then $\xbar^TQ\xbar = \xbar^T\sbar$. Further, $\xbar^T\sbar \geq 0$.
\end{lemma}
\begin{proof}
Substitute $(\xbar,\lambdabar,\sbar)$ in~\eqref{hqp1.rd} and multiplying by $\xbar^T$ obtain
\begin{align*}
    \xbar^Tr_d(\xbar,\lambdabar,\sbar) = \xbar^TQ\xbar + (A\xbar)^T\lambdabar - \xbar^T\sbar = 0.
\end{align*}
Using $A\xbar = 0$ in the above proves the claim. The second claim follows from Assumption~\ref{assume:ZQZ}.
\end{proof}

The proof on complexity (Theorem~\ref{thm:complexity}) relies on 3 key lemmas.  Lemma~\ref{lemma:bndxs} first bounds $\upsilon^k \|(x^k,s^k)\|_1$. Lemma~\ref{lemma:bndstep} provides a bound on scaled search directions. Finally, Lemma~6.7~\cite{WrightIPMBook} ensures that there exists an uniform lower bound on $\alpha^k$.  We first provide a bound on $\nu^k (x^k,s^k)$. 
\begin{lemma}\label{lemma:bndxs}
Suppose the initial iterate satisfies~\eqref{initpoint}. Then for any iterate $(x^k,\lambda^k,s^k)$ 
\begin{align}
    \zeta\upsilon^k \|(x^k,s^k)\|_1 \leq 4\beta (n+1)\mu^k. \label{bndxs}
\end{align}
\end{lemma}
\begin{proof}
\begin{subequations}
Define
\begin{align}
    (\xbar,\lambdabar,\sbar) =&\; \upsilon^k(x^0,\lambda^0,s^0) + (1-\upsilon^k)(x^\star,\lambda^\star,s^\star) \nonumber \\
    &\; - (x^k,\lambda^k,s^k). \label{lem2.defbar}
\end{align}
Then $(\xbar,\lambdabar,\sbar)$ satisfy the conditions in Lemma~\ref{lemma.1} and $\xbar^T\sbar \geq 0$.  Hence,
\begin{align}
    0 \leq&\; \xbar^T\sbar \nonumber \\
    =&\; (\upsilon^k)^2 (x^0)^Ts^0 + (1-\upsilon^k)^2(x^\star)^Ts^\star + (x^k)^Ts^k \nonumber \\
    &\; + \upsilon^k(1-\upsilon^k)((x^0)^Ts^\star + (x^\star)^Ts^0) \nonumber \\
    &\; -\upsilon^k((x^0)^Ts^k + (x^k)^Ts^0) \nonumber \\
    &\; - (1-\upsilon^k) ((x^\star)^Ts^k + (x^k)^Ts^\star ). \nonumber
\end{align}
Using $(x^\star)^Ts^\star = 0$ and $(x^\star)^Ts^k + (x^k)^Ts^\star \geq 0$ in the above and rearranging obtain
\begin{align*}
    &\; \upsilon^k((x^0)^Ts^k + (x^k)^Ts^0) \\
    \leq&\; (\upsilon^k)^2 (x^0)^Ts^0 + (x^k)^Ts^k \\
    &\; + \upsilon^k(1-\upsilon^k)((x^0)^Ts^\star + (x^\star)^Ts^0).
\end{align*}
The above is identical to~\cite[eqn.~(6.20)]{WrightIPMBook} and the arguments in~\cite[Lemmas~6.3-6.4]{WrightIPMBook} apply to yield the result.
\end{subequations}
\end{proof}

We next provide a bound on the scaled search directions.
\begin{lemma}\label{lemma:bndstep}
Suppose the initial iterate satisfies~\eqref{initpoint}. Then there is a constant $\eta > 0$ independent of $n$ such that 
\begin{align*}
    \|(D^k)^{-1}\Delta x^k\| \leq \eta (n+1)\mu^k,\, \|D^k \Delta s^k\| \leq \eta (n+1)\mu^k 
\end{align*}
where $D^k = (X^k)^{\half}(S^k)^{-\half}$.
\end{lemma}
\begin{proof}
\begin{subequations}
Define
\begin{align}
    (\xbar,\lambdabar,\sbar) =&\; (\Delta x^k, \Delta \lambda^k, \Delta s^k) + \upsilon^k (x^0,\lambda^0,s^0)  \nonumber \\
    &\;- \upsilon^k  (x^\star,\lambda^\star,s^\star) \label{lem3.defbar}
\end{align}
Then $(\xbar,\lambdabar,\sbar)$ satisfy the conditions in Lemma~\ref{lemma.1} and $\xbar^T\sbar \geq 0$.  From the last row in~\eqref{iipm.step} 
\begin{align}
    &\; S^k(\Delta x^k + \upsilon^k(x^0 - x^\star)) + 
    X^k(\Delta s^k + \upsilon^k(s^0 - s^\star)) \nonumber \\
    =&\; -X^ks^k + \sigma^k\mu^k e + \upsilon^k S^k(x^0 - x^\star) + 
    \upsilon^k X^k(s^0 - s^\star).
\end{align}
Multiplying through by $(X^kS^k)^{-\half}$ and the definition of $D^k$ obtain
\begin{align}
    &\; (D^k)^{-1}(\Delta x^k + \upsilon^k(x^0 - x^\star)) + 
    D^k(\Delta s^k + \upsilon^k(s^0 - s^\star)) \nonumber \\
    =&\; -(X^kS^k)^{-\half} (X^ks^k - \sigma^k\mu^k e) \nonumber \\
    &\; + \upsilon^k (D^k)^{-1}(x^0 - x^\star) + 
    \upsilon^k D^k(s^0 - s^\star).
\end{align}
Taking norms on both sides, squaring and using $\xbar^T\sbar \geq 0$ obtain
\begin{align}
    &\; \|(D^k)^{-1}(\Delta x^k + \upsilon^k(x^0 - x^\star))\|^2 \\
    &\; + \|D^k(\Delta s^k + \upsilon^k(s^0 - s^\star))\|^2 \nonumber \\    
    \leq&\; \left\{ \begin{aligned}
    \|(X^kS^k)^{-\half}\| \|(X^ks^k - \sigma^k\mu^k e)\| \nonumber \\
    + \upsilon^k \|(D^k)^{-1}(x^0 - x^\star)\| + 
    \upsilon^k \|D^k(s^0 - s^\star)\| \end{aligned}\right\}^2.
\end{align}
Starting from this inequality the arguments in the proofs of Lemmas~6.5-6.6~\cite{WrightIPMBook} can be followed to show the claim.
\end{subequations}
\end{proof}

Lemma~6.7~\cite{WrightIPMBook} provides an uniform bound on $\alpha^k$ and holds without any change in the arguments. 

The claim on polynomial complexity of Algorithm~\ref{algo:iipm} follows from the Lemmas~\ref{lemma:bndxs},\ref{lemma:bndstep} and Lemma~6.7~\cite{WrightIPMBook}.  The arguments are identical to those in~\cite[Theorem~6.2]{WrightIPMBook}.  We state the complexity result without proof.
\begin{theorem}\label{thm:complexity}
Let $\epsilon > 0$ be given. Suppose that the initial iterate satisfies~\eqref{initpoint} and suppose that $\zeta > 0$ satisfies $\zeta^2 \leq \chi/\epsilon^\kappa$ for some constants $\chi, \kappa > 0$.  Then there is an index $K = O(n^2|\log \epsilon|)$ such that the iterates $(x^k,\lambda^k,s^k)$ generated by Algorithm~\ref{algo:iipm} satisfy $\mu^k \leq \epsilon$ for all $k \geq K$.
\end{theorem}

\section{Numerical Experiments}\label{sec:results}

We present preliminary numerical experiments with the HQP formulation using the IPM solver \ipopt~\cite{Ipopt}.  \ipopt\, implements the standard IPM for QPs i.e. the algorithm cannot determine a certificate of infeasibility. In our experiments, we generated random instances of QP~\eqref{qp} with a single equality constraint with $C = I_n$, $c = e$, $E \in \R^{1 \times n}$, $f = -1$. The coefficients in the single equality are restricted to be nonnegative and generated at random.  It is easy to deduce that there exists no $ y \geq 0 $ such that $Ey < 0$ for $E \geq 0$. Hence, the instances are all infeasible. We presented the QP formulation directly to \ipopt.  The HQP formulation ~\eqref{hqp} is derived for each such random instance and is also presented for solution to \ipopt.  We generated 10 random instances of different sizes $n \in \{10,25,50\}$. 

In the first set of experiments, we set the \ipopt\, option \texttt{mehrotra\_algorithm=no}.  In this case, the predictor-corrector algorithm of Mehrotra is disabled.  When solving the QP formulation, \ipopt\, always terminates after 8-10 iterations with the indication that the problem might be infeasible.  When solving the HQP formulation, \ipopt\, always terminates with an optimal solution of $0$ as specified in Theorem~\ref{thm:infeasible} in about 16-20 iterations.

In the second set of experiments, we enabled the predictor-corrector algorithm in \ipopt\, by setting \texttt{mehrotra\_algorithm=yes}.  In this setting \ipopt\, hits the iteration limit on every single instance when solving the QP formulation. The dual infeasibility blows up to infinity in every instance. When solving the HQP formulation \ipopt\, always terminates with an optimal solution of $0$ in about 8-10 iterations.

In the third set of experiments, we set $f = 1$.  In this case the QP instances are always feasible.  We verified that both the QP and HQP formulations solved the problem to optimality.  Further, the solutions to QP can be recovered from the HQP could be recovered according to Theorem~\ref{thm:optimal}.  Thre is no significant difference in the number of iterations taken for convergence using either formulation. However, the computational time per iteration for HQP is higher than that of QP formulation.  This is likely due to the matrix in~\eqref{iipm.step} being dense. This can be rectified using a tailored Schur-complement-based implementation for the step computation.  We will investigate this in a future work.

\section{Conclusions \& Future Work}\label{sec:conclusions}

We presented a homogeneous formulation of QP that allows for robust detection of infeasibility in QPs. We also presented an infeasible IPM for the solution of the HQP and showed polynomial iteration complexity. In the context of IPMs, the paper leaves open a number of future avenues for exploration.  Firstly, the sparsity of the linear systems in the step computation of IIPM are now affected by the sparsity of $c, f$ which can be dense. To effectively handle this a tailored linear algebra solution is required. Secondly, the linear system $E$ may have structure such as in the case of MPC.  Exploiting this through a decomposition approach such as in~\cite{RaoWrightRawlings1998} is critical to improving the computational efficiency. Thirdly, the user of predictor-corrector steps will be crucial for improving the practial performance. We will explore these in a future work. The homogeneous formulation can be readily presented to any QP algorithm.  We will also investigate the performance of other QP algorithms on the proposed formulation. 

\bibliographystyle{IEEEtran}
\bibliography{refs}

\end{document}